\pdfoutput=1

\documentclass{amsart}
\usepackage{latexsym}
\usepackage{hyperref}
\usepackage{amsmath,amsfonts,epsfig, graphics, graphicx, amsthm, amssymb}
\input{xy}
\xyoption{all}

\newcommand\coloneqq{:=}

\newcommand\R{\mathbb{R}}

\newtheorem{theorem}{Theorem}[section]
\newtheorem{lemma}[theorem]{Lemma}
\newtheorem{corollary}[theorem]{Corollary}
\newtheorem{varthm1}{Main Theorem}

\newtheorem*{Th}{Theorem}
\newtheorem*{theorem67}{Polymath 8b, Theorem 6.7}

\begin{document}
\title{Prime Tuples and Siegel Zeros}
\author[T. Wright]{Thomas Wright}

\begin{abstract}
Under the assumption of infinitely many Siegel zeroes $s$ with $Re(s)>1-\frac{1}{(\log q)^{R}}$ for a sufficiently large value of $R$, we prove that there exist infinitely many $m$-tuples of primes that are $\ll e^{1.9828m}$ apart.  This "improves" (in some sense) on the bounds of Maynard-Tao, Baker-Irving, and Polymath 8b, who found bounds of $e^{3.815m}$ unconditionally and $me^{2m}$ assuming the Elliott-Halberstam conjecture; it also generalizes a 1983 result of Heath-Brown that states that infinitely many Siegel zeroes would imply infinitely many twin primes.  Under this assumption of Siegel zeroes, we also improve the upper bounds for the gaps between prime triples, quadruples, quintuples, and sextuples beyond the bounds found via Elliott-Halberstam.
\end{abstract}
\maketitle

\section{Introduction}

One of the most exciting recent breakthroughs in number theory has been the much-celebrated progress toward the twin prime and Polignac conjectures.  The current wave of excitement was sparked by Zhang \cite{Zh}, who proved in 2013 that
$$H_1\leq 70,000,000,$$
where $H_m$ is the smallest number such that there are infinitely many intervals of length $H_m$ containing at least $m+1$ primes.  Maynard and Tao \cite{Ma} were later able to establish a slightly different method that would eventually (with the help of Polymath 8b \cite{Po}) reduce this bound to
$$H_1\leq 246.$$
Moreover, Maynard and Tao (as well as Baker and Irving \cite{BI}, Stadlmann \cite{St}, and the Polymath project) were able to bound $H_m$ for larger values of $m$ as well; they gave bounds for $m=2$ to 5, as well as a general upper bound for $H_m$.  Those results have been iteratively improved from the initial work of Maynard and Tao - we list here the best current results, as found in \cite{St}:
\begin{gather*}
H_2\leq 396,516,\\
H_3\leq 24,407,016,\\
H_4\leq 1,391,051,532,\\
H_5\leq 77,510,685,234,\\
H_m\ll e^{3.8075m}.
\end{gather*}

There has also been a great deal of work done on these results under the assumption of certain conjectures.  In particular, if one assumes the Elliott-Halberstam conjecture, one can find even better results (the bounds below are listed in \cite{Po}, though the $H_1$ bound first appears in \cite{Ma}):
\begin{gather*}
H_1\leq 12,\\
H_2\leq 270,\\
H_3\leq 52,116,\\
H_4\leq 474,266,\\
H_5\leq 4,137,854,\\
H_m\ll me^{2m}
\end{gather*}
If one assumes the generalized Elliott-Halberstam conjecture, one can make even further progress on the bounds on the first two cases (see \cite{Po}), finding that
\begin{gather*}
H_1\leq 6,\\
H_2\leq 252;
\end{gather*}
though there is of yet no progress on cases with four or more primes.

Note that the conjectures mentioned above assume that primes generally behave well.  The Elliott-Halberstam conjecture assumes that the primes up to $x$ behave nicely mod $q$ (on average) for all $q<x^\theta$ as long as $\theta<1$; the generalized Elliott-Halberstam conjecture goes even further, assuming that all sequences that obey certain conditions behave nicely mod $q$ (on average) for $q<x^\theta$ with $\theta<1$.


\section{Introduction: Siegel Zeroes}

There are, however, other countervailing conjectures that assume that primes do \textit{not} behave nicely in certain circumstances.  One of the most famous is the assumption of Landau-Siegel zeroes, which is explained as follows.

For a character $\chi_q$ modulo $q$, we define as usual

$$L(s,\chi_q)=\sum_{n=1}^{\infty}\frac{\chi_q(n)}{n^s}.$$

In 1918, Landau \cite{La} proved the following:
\begin{Th} (Landau 1918)
There exists a constant $c$ such that for any $q$ and any real-valued character $\chi_q$ mod $q$, there exists at most one $s=\sigma+it$ for which $L(s,\chi_q)=0$ and $$Re(s)>1-\frac{c}{\log q(|t|+2)}.$$.
\end{Th}

Such a zero, if it is to exist, is known as a \textit{Siegel zero}, a \textit{Landau-Siegel zero}, or an \textit{exceptional zero}, and the associated character is known as an \textit{exceptional character}.  It is known that if such a zero exists, it must be real (i.e. $t=0$) and the character must be real-valued as well.

We note that in our definition above (and indeed in the literature about the topic), the definition of a Siegel zero is a bit ill-defined in that it depends on the choice of a constant $c$ in the equation above.  Generally, at the beginning of a paper assuming the existence of Siegel zeroes, the authors must declare which specific definition of Siegel zero they will require for their theorem; we will do so here shortly.

Before we do that, however, we recall the two theorems that give us the best known upper bounds for Siegel zeroes.  The first is the best effective bound; the second is the best ineffective bound:

\begin{Th} (Landau 1918)
There exists an effectively computable constant $C$ such that for any $q$ and any real-valued character $\chi_q$ mod $q$, a Siegel zero $\beta$ must satisfy
$$\beta<1-\frac{C}{q^\frac 12 \log^2 q}.$$.
\end{Th}

\begin{Th} (Siegel 1935)
For any $\epsilon$ there exists a constant $C(\epsilon)$ such that
$$\beta<1-C(\epsilon)q^{-\epsilon}.$$
\end{Th}

Of course, mathematicians do not expect such a zero to exist for any $q$ or $\chi_q$; in fact, the generalized Riemann hypothesis predicts all non-trivial zeroes will have $Re(s)=\frac 12$.  Nevertheless, we cannot yet rule out the existence of Siegel zeroes.  In fact, the world of Siegel zeros has been described as ``an `illusory' or `ghostly' parallel universe in analytic number theory that should not actually exist, but is surprisingly self-consistent and to date proven to be impossible to banish from the realm of possibility\footnote{This quote appears in Terence Tao's blog article ``Heath-Brown’s theorem on prime twins and Siegel zeroes," published 26 August 2015.}."

We reiterate that if such an $s$ and $\chi_q$ exist with $\chi_q$ being a primitive character mod $q$, $\chi_q$ must also be a real-valued character mod $q$.  This means that there must exist a fundamental discriminant $M|q$ such that $\chi_q(n)$ is derived from the Kronecker symbol $\left(\frac Mn\right)$.




Surprisingly, the existence of a Siegel zero often yields unusually nice results, including results related to prime gaps.  Perhaps the nicest example of such a result is the following, which appears in \cite{HB83}:

\begin{Th} (Heath-Brown 1983) Let $\chi_q$ be a character mod $q$, and let $\beta$ be such that $L(\beta,\chi_q)=0$ with
$$\beta>1-\frac{1}{3\log q}.$$
If there are infinitely many $\chi_q$ for which such a $\beta$ exists then there are infinitely many twin primes.
\end{Th}

In particular, Heath-Brown found that for any such $q$, the number of twin primes $(p,p+2)$ with $q^{250}\leq p\leq q^{500}$ is as conjectured in the Hardy-Littlewood conjectures\footnote{These bounds on $q$ have since been widened by Tao and Ter\"{a}v\"{a}inen to $q^{20.5+\epsilon}\leq p\leq q^{\eta^{\frac 12}}$, where $\eta=\frac{\log q}{1-\beta}$ \cite{TT}.}.

Unfortunately, Heath-Brown's methods do not appear to generalize to tuples beyond pairs.  However, with the new results on prime tuples, it seems natural to ask whether Siegel zeroes might be of help in the more general prime tuple case.  This is the motivating question of our current paper.

\section{Main Result}

We assume a somewhat stronger condition on the exceptionality of the Siegel zeroes to prove the following:


\begin{varthm1}
Fix an $A>2$, and let $r=554,401$.  Assume that there are infinitely many $D$ and $\chi_D$ such that for each $D$, there exists a real $s_D$ with $L(s_D,\chi_D)=0$
and $$s_D>1-\frac{1}{(\log D)^{r^r+A}}.$$
Then for any $m\geq 1$,
$$H_m\ll e^{1.9828m}.$$
\end{varthm1}

For smaller values of $m$, we also find the following:
\begin{varthm1} With the same assumptions about Siegel zeroes as given in the previous theorem, we have
\begin{gather*}
H_2\leq 264,\\
H_3\leq 49,192,\\
H_4\leq 439,812,\\
H_5\leq 3,775,860.
\end{gather*}
\end{varthm1}
We note that all of these are better results than the assumption of Elliott-Halberstam, although our $H_2$ result is still a lesser bound than the one given by generalized Elliott-Halberstam.  In the case of $m=1$, our methods give $H_1\leq 12$, which, while on par with Elliott-Halberstam, is obviously majorized by Heath-Brown's result and hence is not further discussed here.

\section{Methods}

The key result we will use will be one of Friedlander and Iwaniec \cite{FI}, who proved the following:


\begin{theorem}\label{FIFI} (Friedlander-Iwaniec 1992)
Let $\chi_D$ be a real character mod $D$.  Let $x>D^r$ with $r=554,401$, let $q<x^{\frac{233}{462}}$, and let $(a,q)=1$.  Then
\begin{gather}\label{eqn41formula}\pi(x,q,a)=\frac{\pi(x)}{\phi(q)}\left(1-\chi_D\left(\frac{aD}{(q,D)}\right)+O\left(L(1,\chi_D)(\log x)^{r^r}\right)\right),\end{gather}
where the constant is absolute and computable.
\end{theorem}
We note here that while Friedlander and Iwaniec claim this theorem for $q<x^{\frac{233}{462}}$, their work actually proves this theorem for $$q<x^{\frac{58\left(1-\frac 1r\right)}{115}}=x^{.5043469...},$$which is a slightly better bound (since $233/462=.504329...$).  Friedlander and Iwaniec simply chose the above to simplify the statements of their main theorems; however, this additional leeway will slightly improve our results\footnote{See \cite{FI}, Page 2049 for more information; Friedlander's and Iwaniec's simplification of results occurs between (5.11) and Proposition 5.2. Friedlander and Iwaniec themselves note that they have shrunk the bounds slightly to make the result easier to state.}.

Now, Theorem \ref{FIFI} is true for any choice of $\chi_D$ or $D$, but it is only nontrivial when $L(1,\chi_D)$ is small, which means that the character must be exceptional.  Note that if $s$ is a Siegel zero then $L(1,\chi_D)\ll (1-Re(s))\log^2 D$ by mean value theorem (since $L'(s,\chi_D)\ll \log^2 D$ on the interval between the Siegel zero and 1), so the hypothesis on Siegel zeroes in the Main Theorem is sufficient to prove that the error term here is less than the main term.

In their paper, Friedlander and Iwaniec highlight the fact that if $q$ and $D$ are relatively prime, one can find equidistribution of primes in arithmetic progressions mod $q$ as long as $q<x^{\frac{233}{462}}$, which is a better result than the $x^{\frac 12-\epsilon}$ that is achievable under GRH. In our case, however, the more interesting result is what happens when $D|q$ and $\chi_D(a)=-1$:

\begin{corollary} \label{HI} Fix $A$ and $r$ as in Main Theorem 1, and assume that the conditions on Siegel zeroes in that theorem are satisfied as well.  Let $D|q$ and $\chi_D\left(a\right)=-1$.  If $q<x^{\frac{58}{115}\left(1-\frac 1r\right)}$ then
$$\pi(x,q,a)=\frac{2\pi(x)}{\phi(q)}+O\left(\frac{\pi(x)}{\phi(q)(\log x)^{A-2}}\right)$$
\end{corollary}

In other words, the number of primes is double what one would normally expect in a congruence class.  This is the key insight that will allow our main theorems to be proven.

In order to explain this insight, let us define $\theta'$ to be the level of distribution of the primes, which is to say the supremum of the exponents $t$ for which
$$\sum_{q\leq x^{t}}\max_{(a,q)=1}\left|\pi(x,q,a)-\frac{\pi(x)}{\phi(q)}\right|\ll \frac{x}{\log^B x}$$
for every $B>0$.

In \cite{Ma} and \cite{Po}, the authors define a quantity $M_k$ as the supremum of the quotient of two integrals (see (\ref{Ik})-(\ref{Mk}) below for a more specific definition).  For practical purposes, however, this quantity gives information about the ratio between $S_2$, the sieved sum over the primes, and $S_1$, the sum of the sieve weights themselves.  The work of \cite{Ma} establishes that for us to find tuples of $m+1$ primes in intervals of length $k$, we require
$$M_k>\frac{2m}{\theta'}.$$
In our case, let $\theta$ denote the largest value for which (\ref{eqn41formula}) holds for every $q\leq x^{\theta}$.  If we choose all of our primes from a congruence class $a$ modulo $D$ for which $\chi(a)=-1$ then by the corollary above, we will have double the expected number of primes; thus, the size of $S_2$ is also doubled, and hence we only require
$$M_k>\frac{m}{\theta}.$$
While our $\theta$ isn't as large as $\theta'$ would be under Elliott-Halberstam (which allows any $\theta'<1$), the excess of primes has the same effect as would doubling the value of $\theta'$.  Since our $\theta$ is also greater than 1/2 according to Theorem \ref{FIFI}, this means that our results actually surpass those assuming Elliott-Halberstam.

\section{Remarks}

Obviously, no one has made conjectures as to what the largest possible value for $\theta$ might be for (\ref{eqn41formula}) or Corollary \ref{HI} to hold, since most mathematicians do not believe there are Siegel zeroes at all.  In the paper of Friedlander and Iwaniec, the authors encounter an obstruction at $\theta=\frac{58}{115}(1-\frac 1r)$ from the error term and then another one at $\theta=2/3$ from the main term.  The authors remark that while overcoming the first bound might be tractable, the second one seems a bit more difficult to improve.  The reason for this is that the first bound is derived from a paper on ternary divisor sums \cite{FIt} where the main result has since been improved; the improved version is not yet known to be amenable to the methods of \cite{FI}, but it is not hard to imagine that one could find a way to apply the improved result.  The second bound, however, comes from the famous Weil bound for Kloosterman sums; as such, any improvement over 2/3 would likely require either a method of estimating the main term that avoids using this bound or else an improvement on the Weil bound itself.


On the other hand, conjectures about primes in arithmetic progressions (such as Montgomery's conjecture on arithmetic progressions \cite{Mo1}, \cite{Mo2} or Elliott-Halberstam) say that the equidistribution of primes in classes should apply modulo $q<x^{\theta'}$ for every $\theta'<1$.  Let's imagine for a moment that one could make a similar conjecture about \textit{non}-equidistribution of primes mod $q$ in the presence of a Siegel zero; in other words, assume that (\ref{eqn41formula}) and Corollary \ref{HI} could be proven for $q<x^{\theta}$ for every $\theta<1$, and assume that there are infinitely many Siegel zeroes as required by the Main Theorems.  Then one would find the following\footnote{The numbers on this list come from \cite{Po}, Table 3, where the numbers chosen are simply the lowest $k$ for which $M_k>m+1$.} bounds for $H_m$:


\begin{gather*}
H_1=2\\
H_2\leq 12,\\
H_4\leq 270,\\
H_6\leq 52,116,\\
H_m\ll e^{(1+\epsilon') m},
\end{gather*}
where the last line holds for any $\epsilon'>0$.


The result for $H_1$ is notable here, since it would be a resolution of the twin prime conjecture.  Note that this result for $H_1$ does not require $\theta$ to go all the way up to 1; if the distribution held for $\theta=.722$, we could recover the twin primes result of Heath-Brown \cite{HB83}.

\section{Outline}
The proof, for the most part, will follow the framework of Maynard and Tao.  The differences are as follows:

1.) Obviously, instead of invoking Bombieri-Vinogradov or Elliott-Halberstam, we invoke the result of Friedlander and Iwaniec.  Since we have an explicit error term mod $q$ for each $q$, we are not required to do any fancy Cauchy-Scwhartz machinations; we can simply plug in an estimate for $\pi(x,q,a)$ and then sum this term up over all allowable choices of $q$.


2.) If we wish to ensure that the number of primes is double what Dirichlet's theorem would predict, we must shift the prime tuple so as to put all of the terms in classes $a$ mod $q$ for which $\chi_D(a)=-1$ and $D|q$.  Since $D$ is small relative to $x$, this causes little issue in our analysis.  This shifting is covered in section \ref{shift}.

From here, the proof mostly follows the ideas described in \cite{Ma}, albeit with $W$ a bit bigger than in the original paper.  We note that \cite{Ka} and \cite{BFM} outline the process for dealing with a larger $W$, and the effect on the final calculations is minimal.


\section{Shifting the Primes}\label{shift}

To begin, let us choose an admissible tuple
$$T=\{n+h_1,\cdots,n+h_k\}.$$


In order to use the Friedlander-Iwaniec result, we must first shift our prime tuple so that all of the terms are in a good class for the character modulo $D$.

We know that an exceptional zero indicates that the character must be real-valued and hence must arise from a Kronecker character.  As such, the character can only take the values of $-1$, 1, and 0.  We will deal only with the case where $\chi_D$ is a primitive character; in other words,
$$\chi_D(n)=\left(\frac{\chi_D(-1)D}{n}\right).$$
We do not need to consider the case when $\chi_D$ is imprimitive, for if it is, it must arise from a primitive character $\chi_{C}$ for some $C|D$.  If that were true, $L(s,\chi_{C})$ would have the same Siegel zero with $$Re(s)>1-\frac{1}{(\log D)^{r^r+A}}>1-\frac{1}{(\log C)^{r^r+A}}.$$  So we may simply take $C$ to be our initial $D$ and proceed as normal.

First, let $g$ be the largest prime divisor of $D$, and let
\begin{gather}\label{DnD}
D=gD'.
\end{gather}
Since $\chi_D$ is primitive, we know that $D$ must be square-free except possibly for a power of 2 that is at most $2^3$.  In particular, this means that $g\gg \log D.$

Now, our character must split up multiplicatively, so
$\chi_D(n)=\chi_g(n)\chi_{D'}(n),$
where $\chi_g(n)$ and $\chi_{D'}(n)$ must both take values in $\{-1,1,0\}$.

In this section, our goal will be to show the following:

\begin{lemma}\label{kron}
For a fixed $k$, let $D$ be such that $\log\log\log(D)\gg k$.  Then there exists an $l\leq D$ such that for all $i$,
$$\chi_D\left(l+h_i\right)=-1.$$
\end{lemma}
\begin{proof}
First, we know that since our tuple is admissible, there exists an $n'$ such that $(n'+h_i,D)=1$ for each $i$.

Now, consider the expression
$$\prod_{i=1}^k \left(1-\chi_D(x+n'+h_i)\right).$$
Clearly, this expression is always non-negative, but it will only equal a non-zero value when all of the $\chi_D(x+n'+h_i)$ are either -1 or 0.

So for $D'$ and $g$ as defined above in (\ref{DnD}), let $x=D'y$ and consider the sum
$$H=\sum_{y=1}^g\prod_{i=1}^k \left(1-\chi_D(D'y+n'+h_i)\right),$$
Note that since $g$ is prime and $(n'+h_i,gD')=1$ for all $i$, $\chi_D(D'y+n'+h_i)$ can only equal 0 if $g|D'y+n'+h_i$; for each choice of $i$, this will occur exactly once as $y$ ranges from 1 to $g$.  So there are exactly $k$ choices of $y$ where any of the $\chi_D(D'y+n'+h_i)=0$.  For each of these choices of $y$,
$$0\leq \prod_{i=1}^k \left(1-\chi_{D}(D'y+n'+h_i)\right)\leq 2^{k-1}.$$
So if $H>k2^{k-1}$ then there must exist a tuple where all of the $\chi_D(D'y+n'+h_i)=-1$, which is as required.

Now, we can rewrite $H$ as
\begin{align*}
H=&\sum_{y=1}^g\prod_{i=1}^k \left(1-\chi_{D'}(D'y+n'+h_i)\chi_{g}(D'y+n'+h_i)\right)\\
=&\sum_{y=1}^g\prod_{i=1}^k \left(1-\chi_{D'}(n'+h_i)\chi_{g}(D'y+n'+h_i)\right)
\end{align*}
For each $i$, let us write $\epsilon_i=\chi_{D'}(n'+h_i)$.  Again, these $\epsilon_i$ are all $\pm 1$ since $(D',n'+h_i)=1$.  So
\begin{align*}
H=&\left|\sum_{y=1}^g\prod_{i=1}^k \left(\epsilon_i-\chi_{g}(D'y+n'+h_i)\right)\right|
\end{align*}
Let us multiply out the product.  So
\begin{align*}
H=&\left|\left(\sum_{y=1}^g\prod_{i=1}^k\epsilon_i\right)+ \left(\sum_{y=1}^g\sum_{j=1}^{2^k-1}\left(\prod_{i\in S_j}\epsilon_i\right)\chi_{g}(Q_j(y))\right)\right|
\end{align*}
where the $Q_j(y)$ are polynomials of degree $\leq k$ that are square-free when viewed over $\mathbb F_g$, and the $S_j$ are distinct subsets of $\{1,2,\cdots,k\}$.  By triangle inequality,
\begin{align*}
H\geq &\left|\sum_{y=1}^g\prod_{i=1}^k\epsilon_i\right|- \left|\sum_{j=1}^{2^k-1}\left(\prod_{i\in S_j}\epsilon_i\right)\sum_{y=1}^g\chi_{g}(Q_j(y))\right|\\
\geq &\left|\sum_{y=1}^g 1\right|- \sum_{j=1}^{2^k-1}\left|\sum_{y=1}^g\chi_{g}(Q_j(y))\right|
\end{align*}

The first absolute value is clearly equal to $g$; for the second, we can use the Weil bound for character sums over $\mathbb F_g$, finding that
$$\left|\sum_{y=1}^g\chi_{g}(Q_j(y))\right|\leq (k-1)\sqrt g.$$
So
$$H\geq g-k2^{k-1}\sqrt g.$$
This is clearly larger than $k2^{k-1}$ for large values of $g$.  So there must exist a $y'\leq g$ such that
$\chi_D(D'y'+n'+h_i)=-1$ for all $i$.  We then set
$$l\equiv D'y'+n'\pmod D$$to complete the theorem.
\end{proof}

\section{The Maynard-Tao Sieve}\label{section8}

Now, we proceed as in the work of Maynard-Tao.  First, in \cite{Ma}, the author introduces a parameter $W$ that eliminates the concern for small primes.  We do something similar here.

Fix a large $L>r$.  For a given $D$, let $X=D^{L}$, and let $R=X^{\frac{\theta}{2}-\epsilon}$ for some $\theta<\frac{58}{115}(1-\frac 1r)$ and $\epsilon>\frac 1L$.  We then define

$$V=\prod_{\substack{2<p<\log\log\log X\\ p\nmid D}}p,$$

and let

$$W=DV.$$

Because our tuple is admissible, we know that there exists a congruence class $v$ mod $V$ such that all of the terms in the tuple are coprime to $V$.  So for $l$ as in Lemma \ref{kron}, define $v_0$ such that

\begin{gather*}
v_0\equiv v\pmod V\\
v_0\equiv l\pmod{D}
\end{gather*}



As is standard, for a well-chosen function $\lambda$, we let

$$S_1=\sum_{\substack{n\equiv v_0\pmod W\\ X\leq n\leq 2X}}\left(\sum_{d_i|n+h_i}\lambda_{d_1,\cdots,d_k}\right)^2,$$
and
$$S_2=\sum_{\substack{n\equiv v_0\pmod W\\ X\leq n\leq 2X}}\sum_{i=1}^k \textbf{1}_{\mathbb P}(n+h_i)\left(\sum_{d_i|n+h_i}\lambda_{d_1,\cdots,d_k}\right)^2.$$

This is essentially the sum that was considered in \cite{Ma}.  The only difference here is that $W$ is of size $X^\epsilon$ for some $\epsilon<\frac 1r$.  However, such alterations have been addressed in \cite{Ka}, \cite{BFM}, and elsewhere; the only difference is a slight change in the allowable level of distribution.

Obviously, it will be helpful if we have a definition for $\lambda$.
To this end, let $S_k$ be the set of all piecewise differentiable functions $F:[0,\infty)^k\to \mathbb R$ where the support of $F$ lies on the simplex
$$\mathcal R_k=\{(t_1,...,t_k):\sum_{i=1}^kt_i\leq 1\}$$
For our purposes, we will choose functions $F\in \mathcal S_k$ and then let
$$\lambda_{d_1,\cdots,d_k}=\mu(d_1)\cdots \mu(d_k)F\left(\frac{\log d_1}{\log R},\cdots,\frac{\log d_k}{\log R}\right).$$
We require a lemma about the behavior of such functions; this lemma is the crux of Maynard and Tao's work, since this lemma allows one to turn the search for prime gaps into a search for the best possible $F$.  Define
$$B=\frac{\phi(W)}{W}\log R,$$
and let
\begin{gather}\label{Ik}
I_k(F)=\int_0^\infty...\int_0^\infty F(t_1,...,t_k)^2 dt_1...dt_k,\\
J_k^{(m)}(F)=\int_0^\infty...\left(\int_0^\infty F(t_1,...,t_k)^2 dt_m\right)dt_1...dt_{m-1}dt_{m+1}dt_{k}.
\end{gather}
In practice (in \cite{Po} and elsewhere), the $F$ are always taken to be symmetric; thus, the $J_k^{(m)}(F)$ are independent of $m$, and hence we will often drop the superscript.  So we have the following:
\begin{lemma}\label{L41}
Let $F$, $\lambda$, $W$, and $X$ be as described above.  Then
$$\sum_{\substack{d_1,\cdots d_k,e_1,\cdots,e_k\\ [d_1,e_1],\cdots,[d_k,e_k],W\mbox{ }coprime}}\frac{\lambda_{d_1,\cdots,d_k}\lambda_{e_1,\cdots,e_k}}{[d_j,e_j]}=(1+o(1))B^{-k}I_k(F)$$
and
$$\sum_{\substack{d_1,\cdots d_{k-1},e_1,\cdots,e_{k-1}\\ [d_1,e_1],\cdots,[d_k,e_k],W\mbox{ }coprime}}\frac{\lambda_{d_1,\cdots,d_{k-1},1}\lambda_{e_1,\cdots,e_{k-1},1}}{\phi([d_j,e_j])}=(1+o(1))B^{1-k}J_k(F).$$
\end{lemma}
\begin{proof}
This is essentially the evaluation of (5.2) and (5.18) of \cite{Ma}.  That paper only deals with very small $W$; however, the method applies easily to larger $W$, and others have done exactly this - see e.g. (5.4) and (5.20) of \cite{Ka}.
\end{proof}

\section{The Sum $S$}

From here, we evaluate the two sums.  Define $\mathcal S_k$ to be the set of all functions $F$ that satisfy the criteria listed above, and define


Let
\begin{gather}\label{Mk}
M_k=\sup_{F\in \mathcal S_k}\frac{\sum_{m=1}^{k}J_{k}^{(m)}(F)}{I_k(F)}.
\end{gather}

Then we have the following:

\begin{theorem}\label{S1S2}
For a given $k$, choose $F\in \mathcal S_k$.  Let $S_1$ and $S_2$ be as above, and let $D$, $X$, and $R$ be as defined at the beginning of Section \ref{section8}.  Then

$$S_1=(1+o(1))\frac XWB^{-k}I_k(F)$$
and
$$S_2=(2+o(1))\frac{X}{\phi(W)\log X}B^{1-k}\sum_{m=1}^{k}J_k^{(m)}(F),$$
where the $o(1)$-term goes to zero as $D\to \infty$.
\end{theorem}

Since $\frac{\log R}{\log X}=\frac{\theta}{2}$, $S_2$ can be rewritten as
$$S_2=(\theta+o(1))\frac{X}{W}B^{-k}\sum_{m=1}^{k}J_k^{(m)}(F).$$
This means that if
$$S=S_2-\rho S_1$$
then
$$S=\left((\theta+o(1))\sum_{m=1}^{k}J_k^{(m)}(F)-\rho I_k(F)\right)\frac{X}{W}B^{-k}.$$

The proof here is similar to the proofs in \cite{Ma} or in any number of papers that are based on \cite{Ma}. We split the proof into two parts, one for each sum, below.

\section{The Sum $S_1$}
First, we evaluate
$$S_1=\sum_{\substack{n\equiv v_0\pmod W\\ X\leq n\leq 2X}}\left(\sum_{d_i|n+h_i}\lambda_{d_1,\cdots,d_k}\right)^2.$$
The usual trick is to expand out the square and reverse the order of summation, giving
$$S_1=\sideset{}{'}\sum_{\substack{d_1,\cdots,d_k\\ e_1,\cdots,e_k}}\lambda_{d_1,\cdots,d_k}\lambda_{e_1,\cdots,e_k}\sum_{\substack{n\equiv v_0\pmod W, X\leq n\leq 2X\\ [d_i,e_i]|n+h_i}}1.$$
where the prime indicates that the sum requires the $[d_1,e_1],\cdots,[d_k,e_k],W$ to be pairwise coprime.  The $d_i$ and $e_i$ must all be coprime to $W$ (by choice of $v_0$); moreover, $([d_i,e_i],[d_j,e_j])=1$ since a prime $p|h_i-h_j$ implies $p|W$.  We can then replace the inside sum with $X/(W\prod [d_i,e_i])+O(1)$.  Since $\lambda\ll 1$ and is only supported on $\prod d_i\leq R$, we have
\begin{align*}&\sideset{}{'}\sum_{\substack{d_1,\cdots,d_k\\ e_1,\cdots,e_k}}\lambda_{d_1,\cdots,d_k}\lambda_{e_1,\cdots,e_k}\left(\frac{X}{W\prod [d_i,e_i]}+O(1)\right)\\
=&
\left(\sideset{}{'}\sum_{\substack{d_1,\cdots,d_k\\ e_1,\cdots,e_k}}\lambda_{d_1,\cdots,d_k}\lambda_{e_1,\cdots,e_k}\frac{X}{W\prod [d_i,e_i]}\right)+O(R).
\end{align*}
Hence the error term is negligible.

For the main term, we pull out the $X$ and $W$ and find
$$\frac{X}{W}\left(\sideset{}{'}\sum_{\substack{d_1,\cdots,d_k\\ e_1,\cdots,e_k}}\frac{\lambda_{d_1,\cdots,d_k}\lambda_{e_1,\cdots,e_k}}{\prod [d_i,e_i]}\right)$$
We then invoke Lemma \ref{L41}, giving us
$$S_1=(1+o(1))\frac XWB^{-k}I_k(F).$$




\section{The Sum $S_2$}

Recall first that by Corollary \ref{HI}, if $D|q$ and $\chi(a)=-1$ we have

$$\pi(x,q,a)=\frac{2\pi(X)}{\phi(q)}+O\left(\frac{\pi(X)}{\phi(q)\log^{A-2}(x)}\right).$$
Now, for $S_2$, let us write
$$S_2^{m}=\sum_{\substack{n\equiv v_0\pmod W\\ X\leq n\leq 2X}}\textbf{1}_{\mathbb P}(n+h_m)\left(\sum_{d_i|n+h_i}\lambda_{d_1,\cdots,d_k}\right)^2.$$
Again, we expand out the square and reverse the order of summation, giving
$$=\sideset{}{'}\sum_{\substack{d_1,\cdots,d_{k-1},1\\ e_1,\cdots,e_{k-1},1}}\lambda_{d_1,\cdots,d_k}\lambda_{e_1,\cdots,e_k}\sum_{\substack{n\equiv v_0\pmod W, n\sim X\\ [d_i,e_i]|n+h_i}}\textbf{1}_{\mathbb P}(n+h_m).$$
Evaluating the inside sum gives
$$S_2^{(m)}=\sideset{}{'}\sum_{\substack{d_1,\cdots,d_k\\ e_1,\cdots,e_k}}\lambda_{d_1,\cdots,d_{k-1}}
\lambda_{e_1,\cdots,e_{k-1}}\left(1+O\left(\frac{1}{\log^A X}\right)\right)\left(
\frac{2\pi(X)}{\phi(W)\prod_{i=1}^k\phi([d_i,e_i])}\right).$$
Pulling the $W$'s and $X$'s out front gives
$$\left((1+o(1))\frac{2\pi(X)}{\phi(W)}\right) \sideset{}{'}\sum_{\substack{d_1,\cdots,d_{k-1}\\ e_1,\cdots,e_{k-1}}}\frac{\lambda_{d_1,\cdots,d_k}\lambda_{e_1,\cdots ,e_k}}{\prod_{i=1}^{k-1}\phi([d_i,e_i])}$$
As in the evaluation of $S_1$, by Lemma \ref{L41} we can simplify the inner sum to
\begin{align*}
S_2^{m}=&(2+o(1))\frac{X}{\phi(W)\log X}B^{1-k}J_k^{m}(F)\\
=&(2+o(1))\left(\frac{\theta}{2}\right)\frac{X}{W}B^{-k}J_k^{m}(F).
\end{align*}

\section{Bounds for $H_m$}
From Theorem \ref{S1S2}
$$S=S_2-\rho S_1=(1+o(1))\left((\theta-2\epsilon)\sum_{m=1}^{k}J_k^{m}(F)- \rho I_k(F)\right)\frac XWB^k$$
Let $M_k$ be as defined in (\ref{Mk}).  In order to prove that there exist infinitely many $\rho+1$-tuples, it will then suffice to show that
$$(\theta-2\epsilon)M_k>\rho.$$
The work of Friedlander and Iwaniec allows any $$\theta<\frac{58}{115}\left(1-\frac 1r\right)=.504346916...$$Moreover, since we can choose $L$ to be as large as we like, we can similarly choose epsilon to be as small as we like.  So we will require
$$M_k>\frac{\rho}{.504346916}.$$
From the work of Maynard and Tao \cite{Ma}, it is known that
$$M_k\geq \log k-2\log \log k-2.$$
Polymath 8b \cite{Po} improved this bound to
$$M_k\geq \log k-C$$
for some constant $C$.

For our purposes, it does not matter which bound we use, since setting either one greater than $\frac{\rho}{.504346916}$ will yield
$$k\gg e^{(1.98276)\rho}.$$
From equation (149) in \cite{Po}, the minimum diameter $H(k)$ of a $k$-tuple can be bounded by
$$H(k)\leq k\log k+k\log\log k-k+o(k),$$
So
$$H_m\ll e^{1.9828m}.$$
which is Main Theorem 1.


\section{Bounds for $H_m$ for $m=3,4,5$}

For small values of $m$, we can use the ideas of the Polymath 8b \cite{Po} project to find our bounds for $M_k$.  First, for $m=3$, 4, and 5, we can use Theorem 6.7 in the paper mentioned above:

\begin{theorem67} Let $k \geq 2$, and let $c,T,\tau > 0$ be parameters.  Define the function $g: [0,T] \to \R$ by
\begin{equation*} g(t) \coloneqq \frac{1}{c + (k-1) t}
\end{equation*}
and the quantities
\begin{align*}
m_2 &\coloneqq \int_0^T g(t)^2\ dt \\
\mu &\coloneqq \frac{1}{m_2} \int_0^T t g(t)^2\ dt \\
\sigma^2 &\coloneqq \frac{1}{m_2} \int_0^T t^2 g(t)^2\ dt - \mu^2.
\end{align*}
Assume the inequalities
\begin{align*}
k\mu &\leq 1-\tau \\
k\mu &< 1-T \\
k\sigma^2 &< (1+\tau-k\mu)^2.
\end{align*}
Then one has
\begin{equation*}
 \frac{k}{k-1} \log k - M_k \leq \frac{k}{k-1} \frac{Z + Z_3 + WX + VU}{(1+\tau/2) (1 - \frac{k\sigma^2}{(1+\tau-k\mu)^2})}
\end{equation*}
where $Z, Z_3, W, X, V, U$ are the explicitly computable quantities
\begin{align*}
Z &\coloneqq \frac{1}{\tau} \int_1^{1+\tau}\left( r\left(\log\frac{r-k\mu}{T} + \frac{k\sigma^2}{4(r-k\mu)^2 \log \frac{r-k\mu}{T}} \right) + \frac{r^2}{4kT}\right)\ dr\\
Z_3 &\coloneqq \frac{1}{m_2} \int_0^T kt \log(1+\frac{t}{T}) g(t)^2\ dt \\
W &\coloneqq \frac{1}{m_2} \int_0^T \log(1+\frac{\tau}{kt}) g(t)^2\ dt\\
X &\coloneqq \frac{\log k}{\tau} c^2 \\
V &\coloneqq \frac{c}{m_2} \int_0^T \frac{1}{2c + (k-1)t} g(t)^2\ dt \\
U &\coloneqq \frac{\log k}{c} \int_0^1 \left((1 + u\tau- (k-1)\mu- c)^2 + (k-1) \sigma^2\right)\ du.
\end{align*}
\end{theorem67}
The theorem works by finding a lower bound not for $M_k$ but for a variant $M_k^{[\alpha]}$, which restricts $\mathcal S_k$ to those functions supported on the simplex
$$\mathcal R_k^{[\alpha]}=\{(t_1,...,t_k)\in [0,\alpha]^k:\sum_{i=1}^kt_i\leq 1\}.$$
Here, the authors set $\alpha=T$.  In some sense, this theorem can be seen to build on Zhang's original idea that one can often derive better results by restricting the analysis to smooth moduli.

More specifically, the authors of \cite{Po} seek to maximize $M_k^{[T]}$ by treating it probabilistically, eventually finding (see (114) of \cite{Po})that for any arbitrary $r>1$,
\begin{gather}\label{probdist}\left(\int_0^{r-X_1-\cdots -X_{k-1}}g(t)^2dt\right)\leq \frac{m_2M_k^{[T]}r}{k}\mathbb P\left(X_1+\cdots +X_k\geq r\right) \end{gather}
where the $X_i$ are independent random variables taking values in $[0,T]$ with probability distribution $\frac{1}{m_2}g(t)^2$.  The values of $\mu$ and $\sigma$ for each variable in this distribution are calculated as in the statement of the theorem, and one can see that the goal is then to choose our $c$, $T$, and $\tau$ that maximizes $M_k{[T]}$ in (\ref{probdist}) above.

\cite{Po} defines the parameters $c,T,\tau$ in terms of new parameters $\beta$ and $\eta$ (and our previously defined $k$), where
\begin{align*}
c &:= \frac{\eta}{\log k} \\
T &:= \frac{\beta}{\log k} \\
\tau &= 1 - k\mu
\end{align*}
\cite{Po} then sets about finding the best choices of $\beta$ and $\eta$ to push $M_k$ over the various magical marks that will guarantee a certain number of primes.

$\beta$ and $\eta$ can be chosen arbitrarily.  To find optimal choices for $\beta$ and $\eta$, we performed a Matlab computation where we began with the $\beta$ and $\eta$ chosen in \cite{Po} for the nearest tuple and then let the variables ``walk" until we found apparent maxima.  The results are below:

\begin{theorem} Given the following values for $\beta$, $\eta$, and $k$, we can find the following for $M_k$:\\

(i) $\beta=.973$, $\eta=.9650$: $M_{5229}\geq 5.9484$

(ii) $\beta=.9433$, $\eta=.9786$: $M_{38,801}\geq 7.9310494$

(iii) $\beta=.9215$, $\eta=.9861$: $M_{284,030}\geq 9.913813$
\end{theorem}

Noting that
\begin{gather*}
5.9484>\frac{3}{.504346916}=5.94828...\\
7.9310494>\frac{4}{.504346916}=7.9310487...\\
9.9138119>\frac{5}{.504346916}=9.9138109...
\end{gather*}
we see that these choices of $k$ are the desired ones.


As in \cite{Po}, we will let $H(k)$ denote the minimal diameter $h_k-h_1$ of an admissible $k$-tuple.  Andrew Sutherland was kind enough to compute the minimal tuples\footnote{The tuples are currently hosted at \par \url{https://math.mit.edu/~drew/admissible_5229_49192.txt}, \par \url{https://math.mit.edu/~drew/admissible_38802_439820.txt}, \par \url{https://math.mit.edu/~drew/admissible_284031_3775886.txt}.  \par Note that for the latter two, the tuples are one term larger than for the results stated above; the bounds we state are then the smallest 38801-tuple and 284030-tuple that can be found on these lists.  Interestingly, the direct calculations of a 38801-tuple and a 284030-tuple ran into issues where the algorithm fell into a ``valley" of sorts and ended with a much larger tuple than the above, so we revert to the linked lists.} for the three numbers above:
\begin{gather*}
H_3\leq H(5229)\leq 49,192,\\
H_4\leq H(38,801)\leq 439,812,\\
H_5\leq H(284,030)\leq 3,775,860.
\end{gather*}
This proves Main Theorem 2 for $H_3$, $H_4$, and $H_5$.

\section{Bounds for $H_m$ for $m=2$}

For small values of $m$, concrete values for $M_k$ have been calculated by the Polymath project.  We skip the $m=1$ case because it is inferior to Heath-Brown's result; if $m=2$, Nielsen's table of results for $M_k$ indicates that
$$M_{53}\geq 3.986213$$
and $3.986213$ is greater than $\frac{2}{.504346916}=3.96552...$.  So assuming the existence of infinitely many strong Siegel zeroes, we have
$$H_2\leq 264,$$
which falls between the $H_2$ found if one assumes EH (270) and that found if one assumes GEH (252).

\section{Acknowledgements}
I would like to thank Andrew Sutherland for his help in computing the tuples and for alerting me to the work of \cite{St}, as well as catching a couple of important discrepancies between \cite{Po} and the Polymath website.  I would also like to thank Beau Christ and Jamie Wright for their technological help.  Finally, I would like to thank the referee for numerous helpful comments and suggestions.
\bibliographystyle{line}

\begin{thebibliography}{[HD]}
\normalsize
\bibitem[BI]{BI} R.C. Baker, A.J. Irving. Bounded intervals containing many primes, Math. Z. 286 (2017), 821–-841.
\bibitem[BFM]{BFM} W.D. Banks, T. Freiberg, J. Maynard. On limit points of the sequence of normalized prime gaps,
Proc. Lond. Math. Soc., (3) 113 (2016), 515–539.
\bibitem[Da]{Da} H. Davenport. On character sums in a finite field, Acta Mathematica 71 (1939), 99--121.
\bibitem[FI03]{FI} J. Friedlander and H. Iwaniec, Exceptional characters and prime numbers in arithmetic progressions, Int. Math. Res. Not. 37 (2003), 2033–-2050.
\bibitem[FI85]{FIt} J. Friedlander and H. Iwaniec, Incomplete Kloosterman Sums and a Divisor Problem, Ann. of Math., Second Series, Vol. 121 (2) (1985), 319--344.
\bibitem[Fo]{Fo} K. Ford, Large prime gaps and progressions with few primes, Riv. Math. Univ.
Parma , vol. 12 (1) (2021), 41–-47.
\bibitem[HB]{HB83} D. R. Heath-Brown, Prime twins and Siegel zeros, Proc. London Math. Soc. (3) 47 (1983), no. 2, 193--224.
\bibitem[Ka]{Ka} D.A. Kaptan, A note on small gaps between primes in arithmetic progressions, Acta Arith., 172 (2016), 351–-375.
\bibitem[La]{La} E. Landau, Über die Klassenzahl imaginär-quadratischer Zahlkörper, Nachr. Ges. Wiss. Göttingen (1918), 285--295.
\bibitem[Ma]{Ma} J. Maynard, Small gaps between primes, Ann. of Math. 181 (2015), 383–-413.
\bibitem[Mo1]{Mo1} H. L. Montgomery, Primes in arithmetic progressions, Michigan Math. J. 17 (1970), 33--39.
\bibitem[Mo2]{Mo2} H. L. Montgomery, Topics in Multipicative Number Theory, Lecture Notes in
Mathematics 227, Springer-Verlag, Berlin (1971).
\bibitem[Po]{Po} D. H. J. Polymath, Variants of the Selberg sieve and bounded gaps between primes, Research in the Mathematical Sciences
1:12 (2014), 1--83.
\bibitem[St]{St} Stadlmann, J.  On primes in arithmetic progressions and bounded gaps between many primes,
	arXiv:2309.00425.
\bibitem[TT]{TT} T. Tao and J. Ter\"{a}v\"{a}inen, The Hardy-Littlewood-Chowla conjecture in the presence of a Siegel zero, preprint, https://arxiv.org/pdf/2109.06291.pdf.
\bibitem[Zh]{Zh} Y. Zhang, Bounded gaps between primes, Annals Math 179 (2014), 1121–-1174.
\end{thebibliography}



\end{document}